\documentclass[oneside, 12pt]{article}
\usepackage[margin=3cm]{geometry}

\usepackage[utf8]{inputenc}
\usepackage[english]{babel}

\usepackage{graphicx}
\usepackage{amsthm}   
\usepackage{amsmath, xspace}  
\usepackage{amssymb}  
\usepackage{bm} 
\usepackage{verbatim} 
\usepackage{enumitem} 
\usepackage{cancel}   
\usepackage{multicol}

\usepackage{hyperref}
\hypersetup{
  colorlinks=true,
  citecolor=black,
  linkcolor=black,
  urlcolor=black,
  breaklinks=true,
  bookmarksnumbered=true,
  bookmarksopen=true,
  pdftitle={Quantified Reflection Calculus with one modality},
  pdfauthor={Ana de Almeida Borges and Joost J. Joosten},
  pdfsubject={},
  pdfcreator={Ana de Almeida Borges},
  pdfkeywords={}
}
\usepackage[numbers]{natbib}   

\theoremstyle{definition}
\newtheorem{theorem}{Theorem}[section]
\newtheorem{definition}[theorem]{Definition}
\newtheorem{lemma}[theorem]{Lemma}
\newtheorem{corollary}[theorem]{Corollary}

\newtheorem{remark}[theorem]{Remark}

\makeatletter
\newcommand{\inspect}[1]{%
  \def\inspectspace{\mskip3mu\relax}
  \sbox\z@{$#1$}
  \sbox\tw@{\thickmuskip=0mu$#1$}
  \ifdim\wd\tw@<\wd\z@ \def\inspectspace{\mskip9mu\relax}\fi 
}
\makeatother

\DeclareMathOperator{\Existsa}{\exists}
\newcommand{\Exists}[1]{
  \inspect{#1}
  \Existsa #1 \inspectspace
}

\DeclareMathOperator{\Foralla}{\forall}
\newcommand{\Forall}[1]{
  \inspect{#1}
  \Foralla #1 \inspectspace
}

\newcommand{\F}{\mathcal{F}}
\newcommand{\M}{\mathcal{M}}

\newcommand{\la}{\langle}
\newcommand{\ra}{\rangle}
\newcommand\rest[1]{|_{#1}}

\newcommand{\PA}{\mathsf{PA}}

\newcommand{\gl}{\mathsf{GL}}
\newcommand{\GLP}{\mathsf{GLP}}
\newcommand{\RC}{\mathsf{RC}}
\newcommand{\QRC}{\mathsf{QRC_1}}

\newcommand{\isig}[1]{{\ensuremath {\mathrm{I}\Sigma_{#1}}}\xspace}

\newcommand{\Con}{\text{Con}}

\newcommand{\fv}{\normalfont{\text{fv}}} 
\newcommand{\mdepth}{\normalfont{\text{d}}_\Diamond} 
\newcommand{\udepth}{\normalfont{\text{d}}_\forall} 
\newcommand{\Cl}{\mathcal{C}\ell} 

\renewcommand{\vec}[1]{\bm{#1}}

\newcommand{\gnum}[1]{\ulcorner #1 \urcorner} 

\newcommand{\xaltern}[1]{\sim_{#1}} 

\newcommand{\subst}[2]{[#1 {\leftarrow} #2]} 
\newcommand{\pair}[2]{\la #1, #2 \ra}

  \title{Quantified Reflection Calculus with one modality}
  \author{Ana de Almeida Borges\thanks{\url{anadealmeidagabriel@ub.edu}} \and Joost J. Joosten\thanks{\url{jjoosten@ub.edu}}}
  \date{Universitat de Barcelona\\\vspace{6mm} March 2020}
  
\begin{document}

\maketitle

  \begin{abstract}

    This paper presents the logic $\QRC$, which is a strictly positive fragment of quantified modal logic. The intended reading of the diamond modality is that of consistency of a formal theory. Predicate symbols are interpreted as parametrized axiomatizations. We prove arithmetical soundness of the logic $\QRC$ with respect to this arithmetical interpretation.

    Quantified provability logic is known to be undecidable. However, the undecidability proof cannot be performed in our signature and arithmetical reading. We conjecture the logic $\QRC$ to be arithmetically complete. This paper takes the first steps towards arithmetical completeness by providing relational semantics for $\QRC$ with a corresponding completeness proof. We also show the finite model property, which implies decidability.

  \end{abstract}

  \textbf{Keywords:}
    Provability logic,
    strictly positive logics,
    quantified modal logic,
    arithmetic interpretations,
    feasible fragments.

\section{Introduction}

We present a new provability logic $\QRC$, standing for Quantified Reflection Calculus with one modality. The best known provability logic is perhaps $\gl$ \cite{Boolos:1993:LogicOfProvability}. Recall that $\gl$ is a PSPACE decidable propositional modal logic where the modality $\Box$ is used to model formal provability in some base theory such as Peano Arithmetic ($\PA$). Likewise, the dual modality $\Diamond$ is used to model consistency over the base theory. By Solovay's celebrated completeness result \cite{Solovay:1976} we know that, in a sense, the logic $\gl$ generates exactly the provably-in-$\PA$ structural behavior of formal provability.

Let us make this slightly more precise. By a \emph{realization} $\star$  we mean a map from propositional variables to sentences in the language of Peano Arithmetic. The realization is extended to all propositional modal formulas by defining $(\varphi \wedge \psi)^\star := \varphi^\star \wedge \psi^\star$ and likewise for other Boolean connectives. Finally, we define $(\Box \varphi)^\star := \Box_{\PA} \varphi^\star$, where $\Box_{\PA}$ is a formula in the language of $\PA$ that arithmetizes formal provability in $\PA$ in the sense\footnote{We refrain from distinguishing a formula $\varphi$ from its Gödel number $\gnum{\varphi}$.} that $\PA\vdash \chi$ if and only if $\mathbb N \models \Box_\PA \chi$. We can now paraphrase Solovay's result as $\gl = \{ \varphi \mid \Forall{\star} \, \PA \vdash \varphi^\star \}$.

After Solovay's completeness theorem, it was natural to ask whether one could find a logic that generates exactly the provably-in-$\PA$ structural behavior of formal provability for (relational) \emph{quantified modal logic}. The main difference with $\gl$ is that we now understand a realization $*$ as a map from relation symbols to sentences in the language of Peano Arithmetic such that the free variables match the arity of the relation symbol. Vardanyan showed in \cite{Vardanyan1986} that the situation is now completely different; now  $\{ \varphi \mid \Forall{*} \, \PA \vdash \varphi^* \}$ is $\Pi^0_2$-complete, a big jump from the PSPACE decidability of $\gl$. 

Visser and de Jonge showed that Vardanyan's result can be extended to a wide range of arithmetical theories and called their paper \emph{No escape from Vardanyan's theorem} \cite{VisserAndDeJonge:2006:NoEscape}. Here we shall take some first steps to indeed find an escape to Vardanyan's theorem. We do so by making two adaptations to the standard setting. First, we resort to a very small fragment of Relational Predicate Modal logic called the \emph{strictly positive fragment}. Second, we slightly change the realizations so that we interpret relation symbols not directly as formulas, but as axiomatizations of theories. As such our study follows a recent development of \emph{strictly positive logics} in general (such as \cite{Kikot2019}) and \emph{reflection calculi} in particular (see  \cite{Dashkov:2012:PositiveFragment}, \cite{Beklemishev:2014:PositiveProvabilityLogic}, and \cite{HermoFernandez:2019:BracketCalculus}).

Japaridze \cite{Japaridze:1988} generalized the logic $\gl$ to a polymodal version called $\GLP$, and Beklemishev \cite{Beklemishev:2005:VeblenInGLP} generalized this further to a transfinite setting yielding $\GLP_\Lambda$, where for each ordinal $\xi < \Lambda$ there is a provability modality $[\xi]$, and larger ordinals refer to stronger provability notions. The logic $\GLP_\omega$ has been successfully used in performing a modular ordinal analysis of $\PA$ and related systems (see \cite{Beklemishev:2004:ProvabilityAlgebrasAndOrdinals}, and more recently \cite{BeklemishevPakhomov:2019:GLPforTheoriesOfTruth}). A key feature in the ordinal analysis is that consistency operators $\la n \ra$ can be interpreted as reflection principles, which are finitely axiomatizable. 

However, an interpretation of limit modalities like $\la \omega \ra$ would require non finitely axiomatizable reflection schemata. One way to overcome this problem is by resorting to what was coined the \emph{Reflection Calculus} \cite{Dashkov:2012:PositiveFragment}, \cite{Beklemishev2012} and its transfinite version $\RC_\Lambda$ \cite{FernandezJoosten:2014:WellOrders}. Reflection calculi only allow strictly positive formulas based solely on propositional variables, a verum constant, consistency operators and conjunctions. As such, the arithmetical realizations as above can be taken to be \emph{arithmetical theories} instead of arithmetical formulas. 

The logic $\QRC$ we present in this paper follows this set-up: we will work with sequents of the form $\varphi \vdash \psi$ where both $\varphi$ and $\psi$ are strictly positive formulas built up from $\top$, predicate symbols, conjunction, universal quantification and the $\Diamond$ modality. The latter will refer to the usual notion of formal consistency and predicate symbols are interpreted as theories parametrized by the free variables.

Independently of the reflection calculi, other strictly positive modal logics were studied because of  their computational desirable properties when compared to their non-strict counterparts (see \cite{Kikot2019} for an example). In this line, the logic $\RC$ can be seen as a PTIME decidable fragment of the PSPACE complete logic $\GLP$ (shown in \cite{Dashkov:2012:PositiveFragment}). If indeed the logic $\QRC$ we present in this paper turns out to be arithmetically complete, this would yield, in a sense, a shift from undecidability ($\Pi^0_2$-complete) to decidability when resorting to a strictly positive reflection fragment.

\section{Quantified Reflection Calculus with one modality}

The \emph{Quantified Reflection Calculus with one modality}, or $\QRC$, is a sequent logic in a predicate modal language that is strictly positive.

Towards describing the language of $\QRC$, we fix a countable set of variables $x_0, \hdots$ (also referred to as $x, y, z$, etc.) and define a signature $\Sigma$ as a set of constants and a set of relation symbols with corresponding arity (we have no function symbols). We use the letters $c, c_i, \hdots$ to refer to constants and the letters $S, S_i, \hdots$ to refer to relation symbols.

  Given a signature, a term $t$ is either a variable or a constant of that signature.
  Both $\top$ and any $n$-ary relation symbol applied to $n$ terms are atomic formulas.
  The set of formulas is the closure of the atomic formulas under the binary connective $\land$, the unary modal operator $\Diamond$, and the quantifier $\forall x$, where $x$ is a variable. Formulas are represented by Greek letters such as $\varphi, \psi, \chi$, etc.

The free variables of a formula are defined as usual. The expression $\varphi\subst{x}{t}$ denotes the formula $\varphi$ with all free occurrences of the variable $x$ simultaneously replaced by the term $t$. We say that $t$ is free for $x$ in $\varphi$ if no occurrence of a free variable in $t$ becomes bound in $\varphi\subst{x}{t}$.

\begin{definition}
  Let $\Sigma$ be a signature and $\varphi$, $\psi$, and $\chi$ be any formulas in that language. The axioms and rules of $\QRC$ are the following:

\begin{multicols}{2}
\begin{enumerate}[label=\upshape(\roman*),ref=\thetheorem.(\roman*)]
  \item $\varphi \vdash \top$ and $\varphi \vdash \varphi$;
  \item $\varphi \land \psi \vdash \varphi$ and $\varphi \land \psi \vdash \psi$;
  \item if $\varphi \vdash \psi$ and $\varphi \vdash \chi$, then\\ $\varphi \vdash \psi \land \chi$;
  \item if $\varphi \vdash \psi$ and $\psi \vdash \chi$, then $\varphi \vdash \chi$;\label{rule:cut}
  \item if $\varphi \vdash \psi$, then $\Diamond \varphi \vdash \Diamond \psi$;\label{rule:necessitation}
    \item $\Diamond \Diamond \varphi \vdash \Diamond \varphi$; \label{ax:transitivity}
    \item $\Diamond \Forall{x} \varphi \vdash \Forall{x} \Diamond \varphi$;
  \item if $\varphi \vdash \psi$, then $\varphi \vdash \Forall{x} \psi$\\($x \notin \fv(\varphi)$); \label{rule:forallR}
  \item if $\varphi\subst{x}{t} \vdash \psi$ then $\Forall{x} \varphi \vdash \psi$\\($t$ free for $x$ in $\varphi$);\label{rule:forallL}
  \item if $\varphi \vdash \psi$, then $\varphi\subst{x}{t} \vdash \psi\subst{x}{t}$\\($t$ free for $x$ in $\varphi$ and $\psi$);\label{rule:term_instantiation}
  \item if $\varphi\subst{x}{c} \vdash \psi\subst{x}{c}$, then $\varphi \vdash \psi$\\($c$ not in $\varphi$ nor $\psi$).\label{rule:constants}
\end{enumerate}
\end{multicols}

If $\varphi \vdash \psi$, we say that $\psi$ follows from $\varphi$ in $\QRC$. When the signature is not clear from the context, we write $\varphi \vdash_\Sigma \psi$ instead.
\end{definition}

We observe that our axioms do not include a universal quantifier elimination. However, this and various other rules are readily available via the following easy lemma.
\begin{lemma}
\label{lem:PQMLconsequences}
  The following are theorems (or derivable rules) of $\QRC$:
  \begin{enumerate}[label=\upshape(\roman*),ref=\thetheorem.(\roman*)]
    \item $\Forall{x} \Forall{y} \varphi \vdash \Forall{y} \Forall{x} \varphi$;
    \item $\Forall{x} \varphi \vdash \varphi\subst{x}{t}$ ($t$ free for $x$ in $\varphi$); \label{lem:instantiation}
    \item $\Forall{x} \varphi \vdash \Forall{y} \varphi\subst{x}{y}$ ($y$ free for $x$ in $\varphi$ and $y \notin \fv(\varphi)$);
    \item if $\varphi \vdash \psi$, then $\varphi \vdash \psi\subst{x}{t}$ ($x$ not free in $\varphi$ and $t$ free for $x$ in $\psi$);
    \item if $\varphi \vdash \psi\subst{x}{c}$, then $\varphi \vdash \Forall{x} \psi$ ($x$ not free in $\varphi$ and $c$ not in $\varphi$ nor $\psi$). \label{item:constants_forall}
  \end{enumerate}
\end{lemma}
\begin{proof}
  \hfill

  \begin{enumerate}[label=\upshape(\roman*)]
    \item Starting from $\varphi \vdash \varphi$, apply Rule~\ref{rule:forallL} twice (noting that a variable is always free for itself), concluding $\Forall{x} \Forall{y} \varphi \vdash \varphi$. Now neither $x$ or $y$ is free in the left-hand-side, so use Rule~\ref{rule:forallR} twice to obtain $\Forall{x} \Forall{y} \varphi \vdash \Forall{y} \Forall{x} \varphi$, as desired.
    \item By Rule~\ref{rule:forallL} applied to $\varphi\subst{x}{t} \vdash \varphi\subst{x}{t}$.
    \item By Rule~\ref{rule:forallR} applied to Lemma~\ref{lem:instantiation}.
    \item Observe that, since $x$ is not free in $\varphi$, we have $\varphi \vdash \Forall{x} \psi$ by Rule~\ref{rule:forallR}. We also have $\Forall{x} \psi \vdash \psi\subst{x}{t}$ by \ref{lem:instantiation}, which is enough by Rule~\ref{rule:cut}.
    \item This is a consequence of Rule~\ref{rule:constants} when $x$ is not free in $\varphi$ and Rule~\ref{rule:forallR}.
  \end{enumerate}
\end{proof}

In order to analyze various aspects of our calculus we define two complexity measures on formulas.

\begin{definition}
  Given a formula $\varphi$, its \emph{modal depth} $\mdepth(\varphi)$ is defined inductively as follows:
  \begin{itemize}
    \item $\mdepth(\top) := \mdepth(S(x_0, \hdots, x_{n-1})) := 0$;
    \item $\mdepth(\psi \land \chi) := \max\{\mdepth(\psi), \mdepth(\chi)\}$;
    \item $\mdepth(\Forall{x} \psi) := \mdepth(\psi)$;
    \item $\mdepth(\Diamond \psi) := \mdepth(\psi) + 1$.
  \end{itemize}
  Given a finite set of formulas $\Gamma$, its modal depth is $\mdepth(\Gamma) := \max_{\varphi \in \Gamma}\{\mdepth(\varphi)\}$.

  The definition of quantifier depth $\udepth$ is analogous except for:
  \begin{itemize}
    \item $\udepth(\Forall{x} \psi) = \udepth(\psi) + 1$; and
    \item $\udepth(\Diamond \psi) = \udepth(\psi)$.
  \end{itemize}
\end{definition}

The modal depth provides a necessary condition for derivability, proven by an easy induction on $\varphi \vdash \psi$.

\begin{lemma}
\label{lem:mdepth}
  If $\varphi \vdash \psi$, then $\mdepth(\varphi) \geq \mdepth(\psi)$.
\end{lemma}

In particular, we get irreflexivity for free as stated in the next result. For other calculi this usually requires hard work via either modal or arithmetical semantics \cite{BeklemishevFernandezJoosten:2014:LinearlyOrderedGLP}, \cite{FernandezJoosten:2013:ModelsOfGLP}, \cite{FernandezJoosten:2018:OmegaRuleInterpretationGLP}.
\begin{corollary}
\label{cor:A|/-<>A}
  For any formula $\varphi$, we have $\varphi \not\vdash \Diamond \varphi$.
\end{corollary}

The following lemma tells us that adding constants to our signature does not strengthen the calculus.
\begin{lemma}
\label{lem:signatures}
  Let $\Sigma$ be a signature and let $C$ be a collection of new constants not yet occurring in $\Sigma$. By $\Sigma_C$ we denote the signature obtained by including these new constants $C$ in $\Sigma$. Let $\varphi, \psi$ be formulas in the language of $\Sigma$. Then, if $\varphi \vdash_{\Sigma_C} \psi$, so does $\varphi \vdash_\Sigma \psi$.
\end{lemma}
\begin{proof}
  This is a standard result and a proof for a calculus similar to ours can be found in Section~1.8 of \cite{Goldblatt2011}. The idea is to replace every constant from $C$ appearing in the proof of $\varphi \vdash_{\Sigma_C} \psi$ by a fresh variable. It can easily be seen that axioms are mapped to axioms under this replacement, and that the rules are also mapped correctly. The most interesting case is that of the generalization of constants rule, because replacing new constants by variables in the premise $\varphi\subst{x}{c} \vdash_{\Sigma_C} \psi\subst{x}{c}$ may leave us unable to apply the same rule. Fortunately the term instantiation rule (Rule~\ref{rule:term_instantiation}) suffices to complete the proof.
\end{proof}

\section{Arithmetical semantics}

In this section we look at the intended arithmetical reading of the logic $\QRC$. We consider mathematical theories in the language $\{ 0, 1, +, \times, \leq, = \}$ of arithmetic. We refer the reader to \cite{HajekPudlak:1993:Metamathematics} for details and definitions. We recall that bounded formulas are those formulas where each quantifier occurs bounded as in $\forall y \leq t$, where $y$ does not occur in $t$. The $\Sigma_1$ formulas are those that arise by existential quantification of bounded formulas. Sets of numbers that can be defined by a $\Sigma_1$ formula are called \emph{c.e.} for \emph{computably enumerable}.

The theory \isig{1} contains the defining axioms for our constants and function symbols, say as in Robinson's arithmetic, and moreover allows induction for $\Sigma_1$ formulas. It is well-known that \isig{1} proves $\Sigma_1$-collection, that is:
  \begin{equation*}
    \Forall{ x {<} z} \Exists{y} \varphi(x,y) \to \Exists{y_0} \Forall{x {<} z} \Exists{y{<}y_0} \varphi(x,y)
    .
  \end{equation*}
For the sake of an easy exposition we shall assume that all the theories we work with extend \isig{1}. By $\tau(x)$ we denote the elementary formula that presents the standard axiomatization of \isig{1}. That is to say, $\mathbb N \models \tau (n)$ if and only if $n$ is the Gödel number of an axiom of \isig{1}.

In the arithmetical interpretation of the propositional logic $\RC$, the propositional variables are mapped to (axiomatizations of) theories, and the conjunction of two theories is interpreted as the union of both theories (corresponding to a disjunction in the sense of either being an axiom of the one or of the other). The arithmetical interpretation of each diamond modality is a consistency notion. 

We will fix a provability predicate $\Box_\alpha \varphi$ formalizing the existence of a Hilbert-style proof, which is a sequence of formulas the last of which is $\varphi$ and such that each element of the sequence is either a logical axiom, an axiom in the sense of $\alpha$, or the result of applying a rule to earlier elements in the sequence. We denote the dual consistency notion by  $\Con_{\alpha}(\psi)$ and sometimes write $\Con_{\alpha}$ instead of $\Con_{\alpha}(\top)$.  The following lemma is standard for $\Sigma_1$ axiomatizations $\alpha$ and the reader can consult \cite{Boolos:1993:LogicOfProvability} for details.

\begin{lemma}\label{theorem:provabilityTriviality}
  For any $\Sigma_1$ formula $\alpha$, we have that
\begin{enumerate}[label=\upshape(\roman*),ref=\thetheorem.(\roman*)]
\item\label{theorem:provabilityTriviality:itemTransitivity}
$\isig{1} \vdash \Con_\alpha ( \Con_\alpha) \to \Con_\alpha$;
\item\label{theorem:provabilityTriviality:itemBarcan}
  $\isig{1} \vdash \Exists{z} \Box_\alpha \varphi \to \Box_\alpha \Exists{z} \varphi$.
\end{enumerate}
\end{lemma}

If we now interpret relation symbols from $\QRC$ as theories (parametrized by the free variables), then a universal quantification (which can be conceived of as an infinite conjunction) will be interpreted as an infinite union/disjunction, that is, an existential quantifier. These observations are reflected in Definition \ref{definition:Realisation} below. 

In this section, we reserve the variables $x_i$ for variables in $\QRC$, and the variables $y_i, z_i$ and $u$ are reserved for the arithmetic language with the understanding that the $y_i$ interpret the $\QRC$-constants $c_i$ and the $z_i$ interpret the $\QRC$-variables $x_i$. The variable $u$ is reserved for (Gödel numbers of) axioms of the theories that we denote.

\begin{definition}\label{definition:Realisation}
  A realization $^*$ takes $n$-ary predicate symbols in the language of $\QRC$ to $(n+1)$-ary $\Sigma_1$-formulas in the language of arithmetic, each representing a set of axioms of theories indexed by $n$ parameters. In particular, a realization $^*$ is such that $S(\vec{c}, \vec{x})^* = \sigma(\vec{y}, \vec{z}, u) $ for some $\Sigma_1$ formula $\sigma$ such that for each concrete numerical values for $\vec{y}, \vec{z}$ we have that $\mathbb{N} \models \sigma(\vec{y}, \vec{z}, u)$ if and only if $u$ is the Gödel number of an axiom of the intended corresponding theory. When we use the vector notation in $S(\vec{c}, \vec{x})^* =\sigma(\vec{y}, \vec{z}, u)$ we understand that $\vec{y}$ matches with $\vec{c}$ and $\vec{z}$ matches with $\vec{x}$, and thus if, say, $y_i$ occurs in $\sigma$, then $c_i$ occurs in $S(\vec{c}, \vec{x})$.
\end{definition}

  We extend a given realization $^*$ to $()^*$ on any formula of $\QRC$ as follows:
  \begin{itemize}
    \item $(\top)^* := \tau(u)$;
    \item $(S(\vec{c},\vec{x}))^* : = S(\vec{c},\vec{x})^* \vee \tau(u)$;
    \item $\big(\psi(\vec{c},\vec{x}) \land \delta(\vec{c},\vec{x})\big)^* := \big(\psi(\vec{c},\vec{x})\big)^* \vee \big(\delta(\vec{c},\vec{x})\big)^*$;
    \item $\big(\lozenge \psi(\vec{c},\vec{x})\big)^* := \tau(u) \lor (u = \gnum{\Con_{(\psi(\vec{c},\vec{x}))^*}})$;
    \item $\big(\Forall{x_i} \psi(\vec{c}, \vec{x})\big)^* := \Exists{z_i} \big(\psi(\vec{c}, \vec{x})\big)^*$.
  \end{itemize}
From now on we omit outer brackets, using the same notation for ${^{*}}$ and $()^*$. This may lead to confusion for predicate symbols, but the context should tell us which reading to use. We fix the notation $\psi(\vec{c}, \vec{x})^* = \psi^*(\vec{y}, \vec{z})$ suppressing mention of $u$ when convenient.

Let $T$ be a c.e.~theory in the language of arithmetic which extends $\isig{1}$.  We define (recall that $\chi^*$ will in general depend on $\vec{y}$ and $\vec{z}$):
  \begin{equation*} 
  \mathcal{QRC}_1(T) = \{\varphi(\vec{c},\vec{x}) \vdash \psi(\vec{c},\vec{x}) \ | \
    \Forall{\, ^*} 
      T \vdash \Forall{\theta} \Forall{\vec{y}} \Forall{\vec{z}} (\square_{\psi^*} \theta \to \square_{\varphi^*} \theta)
  \}
  .
  \end{equation*}
  In the above we assume that all the free variables other than $u$ in $\psi^*\wedge \varphi^*$ are among the $\vec{y}$ and $\vec{z}$. The $\theta$ are sentences without free variables. Furthermore, we stress that all realizations map to $\Sigma_1$ formulas (modulo provable equivalence).
We defer the question of whether $\QRC = \mathcal{QRC}_1(T)$ for any sound c.e.~$T$ containing \isig{1} to a future paper and prove here only the soundness inclusion.

\begin{theorem}[Arithmetical soundness]
$\QRC \subseteq \mathcal{QRC}_1(\isig{1})$.
\end{theorem}
\begin{proof}
  We proceed by (an external) induction on the proof of $\varphi \vdash \psi$. We shall briefly comment on some of the cases. 
The case of the axiom $\varphi \vdash \top$ is clear since by an easy induction on $\varphi$ we van prove that over predicate logic $\varphi^*(\vec{y}, \vec{z}, u) \leftrightarrow \tau(u) \vee \varphi'(\vec{y}, \vec{z}, u)$ for some formula $\varphi'$. The axioms $\varphi \land \psi \vdash \varphi$ are easily seen to be sound since $(\varphi \land \psi)^*= \varphi^* \vee \psi^*$, that is, the formula that defines the union of two axiom sets.

The rule that if $\varphi \vdash \psi$ and $\psi \vdash \chi$, then $\varphi \vdash \chi$ is straightforward but the rule 
that if $\varphi \vdash \psi$ and $\varphi \vdash \chi$, then $\varphi \vdash \psi \land \chi$ is slightly more tricky. To see the soundness, we fix a particular realization $*$ and reason in $\isig{1}$. Inside \isig{1} we fix arbitrary $\vec{y}$, $\vec{z}$ and $\theta$ and assume 
$\Box_{(\psi\wedge\chi)^*(\vec{y}, \vec{z})}\theta$, that is, $\Box_{\psi^*(\vec{y}, \vec{z})\vee\chi^*(\vec{y}, \vec{z})}\theta$. Thus, $\Box_{\psi^*(\vec{y}, \vec{z})}\big( \bigwedge \xi_i \to \theta \big)$ for some collection of axioms $\xi_i$ satisfying $\chi^*(\vec{y}, \vec{z})$. By the induction hypothesis on $\varphi\vdash\psi$ we obtain $\Box_{\varphi^*(\vec{y}, \vec{z})}\big( \bigwedge \xi_i \to \theta \big)$, so that $\Box_{\tau}\big( \bigwedge \xi_i \to (\bigwedge \varphi_j \to \theta ) \big)$ for some collection of axioms $\varphi_j$ satisfying 
$\varphi^*(\vec{y}, \vec{z})$.  Since all the $\xi_i$ satisfy $\chi^*(\vec{y}, \vec{z})$ we conclude $\Box_{\chi^*(\vec{y}, \vec{z})}\big( \bigwedge \varphi_j \to \theta  \big)$. Using now the induction hypothesis on $\varphi \vdash \chi$ we conclude $\Box_{\varphi^*(\vec{y}, \vec{z})}\big( \bigwedge \varphi_j \to \theta  \big)$ whence $\Box_{\varphi^*(\vec{y}, \vec{z})}\theta$ as was to be shown.

We will now see the soundness of the necessitation rule, that is,  if $\varphi \vdash \psi$, then $\Diamond \varphi \vdash \Diamond \psi$. We fix some realization $*$. The induction hypothesis for $\varphi\vdash \psi$ applied to the formula $\bot$ gives us $\isig{1} \vdash \Forall{\vec{y}, \vec{z}} \big(\Box_{\psi^*(\vec{y},\vec{z})}\bot \to \Box_{\varphi^*(\vec{y},\vec{z})}\bot\big)$, whence 
\begin{equation}\label{equation:ConsistencyImplications}
  \isig{1} \vdash
    \Forall{\vec{y}, \vec{z}} \big(
      \Con_{\varphi^*(\vec{y},\vec{z})} \to \Con_{\psi^*(\vec{y},\vec{z})}
    \big)
    .
\end{equation}
Let $\pi$ be the standard proof of this. We reason in $\isig{1}$, fixing parameters $\vec{y}, \vec{z}, \theta$ and assuming $\Box_{(\Diamond \psi)^*(\vec{y}, \vec{z})}\theta$. Since $(\lozenge \psi)^* := \tau(u) \lor (u = \gnum{\Con_{\psi^*(\vec{y},\vec{z})}})$, we conclude $\Box_\tau \big ( \Con_{\psi^*(\vec{y},\vec{z})} \to  \theta \big)$. We combine this proof with the proof $\pi$ of \eqref{equation:ConsistencyImplications} to conclude $\Box_\tau \big ( \Con_{\varphi^*(\vec{y},\vec{z})} \to \theta \big )$, whence $\Box_{(\Diamond \varphi)^*(\vec{y}, \vec{z})}\theta$.

The soundness of the axiom $\Diamond \Diamond \varphi \vdash \Diamond \varphi$ is similar, now using Lemma~\ref{theorem:provabilityTriviality:itemTransitivity} instead of \eqref{equation:ConsistencyImplications}. 

To see the soundness of the axiom $\Diamond \Forall{x_i} \varphi \vdash \Forall{x_i} \Diamond \varphi$ we start by proving a
First Claim:
  \begin{equation}
  \label{eq:first_claim}
    \isig{1}\vdash \Con_{(\Forall{x_i} \varphi)^*(\vec{y}, \vec{z})} \to \Forall{z_i} \Con_{ \varphi^*(\vec{y}, \vec{z})}
    .
  \end{equation}
  To prove this, we reason in \isig{1} and assume $\Exists{z_i} \Box_{\varphi^*(\vec{y}, \vec{z})} \bot$, whence for some number $\zeta$ we have that $\Box_{\varphi^*(\vec{y}, \vec{z})\subst{z_i}{\zeta}} \bot$. Then a slight variation of  Lemma~\ref{theorem:provabilityTriviality:itemBarcan} allows us to see that $\Box_{\Exists{z_i} \varphi^*(\vec{y}, \vec{z})} \bot$, and thus $\Box_{ (\Forall{x_i} \varphi)^*(\vec{y}, \vec{z})} \bot$.
  
  We now prove a Second Claim:
  \begin{equation}
  \label{eq:second_claim}
    \isig{1}\vdash \Box_{(\Forall{x_i} \Diamond \varphi)^*(\vec{y}, \vec{z})} \delta \to \Box_{\tau(u)\vee (u = \gnum{\forall x_i \Con_{\varphi^*(\vec{y}, \vec{z})}})} \ \delta
    .
  \end{equation}
  We observe $(\Forall{x_i} \Diamond \varphi)^*(\vec{y}, \vec{z}) = \Exists{z_i} (\Diamond \varphi)^*(\vec{y}, \vec{z}) = \Exists{z_i} \big( \tau (u) \vee u {=} \gnum{\Con_{\varphi^*(\vec{y}, \vec{z})}}\big)$, the latter being provably equivalent to $\tau(u) \vee \Exists{z_i} \big( u {=} \gnum{\Con_{\varphi^*(\vec{y}, \vec{z})}} \big)$. To prove the Second Claim, we reason in \isig{1} and assume the antecedent $\Box_{(\Forall{x_i} \Diamond \varphi)^*(\vec{y}, \vec{z})} \delta$ fixing some $\vec{y}, \vec{z}, \delta$. Thus, we find a collection of numbers $\zeta_j$ such that 
$\Box_\tau \big( \bigwedge_j \Con_{\varphi^* (\vec{y},\vec{z})\subst{z_i}{\zeta_j}} \to \delta\big)$. 
Clearly, $\Box_\tau \big(  \Forall{z_i} \Con_{\varphi^* (\vec{y},\vec{z})} \to \bigwedge_j \Con_{\varphi^* (\vec{y},\vec{z})\subst{z_j}{\zeta_j}}\big)$, which suffices to prove the Second Claim.

  Let us now go back the soundness of the axiom $\Diamond \Forall{x_i} \varphi \vdash \Forall{x_i} \Diamond \varphi$. We fix $*$, reason in \isig{1}, fix $\vec{y}, \vec{z}, \theta$, and assume $\Box_{(\Forall{x_i} \Diamond \varphi)^*(\vec{y},\vec{z})}\theta$. By the Second Claim and the formalized deduction theorem we get $\Box_\tau (\Forall{z_i} \Con_{\varphi^* (\vec{y}, \vec{z})} \to \theta)$. The First Claim now gives us $\Box_{(\Diamond \Forall{x_i} \varphi)^*(\vec{y}, \vec{z})}\theta$ as was to be shown.

The soundness of the $\forall$ introduction rule on the right, that if $\varphi \vdash \psi$, then $\varphi \vdash \Forall{x_i} \psi$ ($x_i \notin \fv(\varphi)$), is not hard but contains a subtlety. To prove it we fix $*$, reason in \isig{1}, fix $\vec{y}, \vec{z}, \theta$ and assume $\Box_{(\forall x_i \psi)^*(\vec{y},\vec{z})} \theta$. Since $(\forall x_i \psi)^*(\vec{y},\vec{z}) = \exists z_i \psi^* (\vec{y},\vec{z})$ we can find numbers $\zeta_j$ such that $\Box_\tau (\bigwedge_j \psi^* (\vec{y},\vec{z})\subst{z_i}{\zeta_j} \to \theta)$. Now by the induction hypothesis we get $\bigwedge_j \Box_{\varphi^*(\vec{y}, \vec{z})\subst{z_i}{\zeta_j}} \psi^* (\vec{y},\vec{z})\subst{z_i}{\zeta_j}$. Since $x_i\notin \fv(\varphi)$ we have  $\bigwedge_j \Box_{\varphi^*(\vec{y}, \vec{z})} \psi^* (\vec{y},\vec{z})\subst{z_i}{\zeta_j}$. Using $\Sigma_1$-collection we obtain $\Box_{\varphi^*(\vec{y}, \vec{z})} \bigwedge_j \psi^* (\vec{y},\vec{z})\subst{z_i}{\zeta_j}$ from which the required $\Box_{\varphi^*(\vec{y}, \vec{z})}\theta$ follows.

The soundness of the remaining rules is straightforward and boils down to interchanging universal quantifiers. 
\end{proof}

\section{Relational semantics}

There have been several proposals for relational semantics for modal propositional logics, from Kripke \cite{Kripke1963} to many others. Overviews can be found in \cite{HughesCresswell1996} and \cite{Goldblatt2011}. We essentially have first-order models glued together by an accessibility relation. Our interpretation of the universal quantifiers is actualist, which means that $\Forall{x} \varphi$ is true at a world $w$ if and only if $\varphi\subst{x}{d}$ is true at $w$ for every $d$ in the domain of $w$, i.e., for every entity $d$ that exists in that world. It might happen, however, that some other world $u$ has a different domain, and thus that it falsifies $\varphi\subst{x}{e}$ for some specific $e$.

We proceed by defining frames and relational models.

\begin{definition}
  A \emph{frame} $\F$ is a tuple $\la W, R, \{M_w\}_{w \in W}\ra$ where:
  \begin{itemize}
    \item $W$ is a non-empty set (the set of worlds, where individual worlds are referred to as $w, u, v$, etc);
    \item $R$ is a binary relation on $W$ (the accessibility relation); and
    \item each $M_w$ is a finite set (the domain of the world $w$, whose elements are referred to as $d, d_0, d_1$, etc).
 \end{itemize}
 The domain of the frame is $M := \bigcup_{w \in W} M_w$.
\end{definition}

\begin{definition}
  A \emph{relational model} $\M$ in a signature $\Sigma$ is a tuple $\la \F, \{I_w\}_{w \in W},\allowbreak \{J_w\}_{w \in W} \ra$ where:
  \begin{itemize}
    \item $\F = \la W, R, \{M_w\}_{w \in W} \ra$ is a frame;
    \item for each $w \in W$, the interpretation $I_w$ assigns an element of the domain $M_w$ to each constant $c \in \Sigma$, written $c^{I_w}$; and
    \item for each $w \in W$, the interpretation $J_w$ is a function assigning a set of tuples $S^{J_w} \subseteq \wp((M_w)^n)$ to each $n$-ary relation symbol $S \in \Sigma$.
  \end{itemize}
\end{definition}

Even though we interpret the universal quantifiers in the actualist way, we cannot allow the domains of each world to be completely unrelated to each other. This is because we want statements such as the axiom $\Diamond \Forall{x} \varphi \vdash \Forall{x} \Diamond \varphi$ to be sound. This axiom forces us to have \emph{inclusive} frames, which means that if $w$ sees a world $u$, then the domain of $w$ is included or at least embedded in the domain of $u$. We also require that our frames be transitive, for we want the axiom $\Diamond \Diamond \varphi \vdash \Diamond \varphi$ to be sound.
Finally, the interpretation of a constant should indeed be constant throughout (the relevant part of) any useful model. Thus, we introduce the notion of adequate frames and models.

\begin{definition}
  A frame $\F$ is \emph{adequate} if the accessibility relation $R$ is:
  \begin{itemize}
    \item inclusive: if $wRu$, then $M_w \subseteq M_u$; and
    \item transitive: if $wRu$ and $uRv$, then $wRv$.
  \end{itemize}
  A model is \emph{adequate} if it is based on an adequate frame and it is:
  \begin{itemize}
    \item concordant: if $wRu$, then $c^{I_w} = c^{I_u}$ for every constant $c$.
  \end{itemize}
  Note that in an adequate and rooted model the interpretation of the constants is the same at every world.
\end{definition}

In order to define truth at a world in a first-order model, we use assignments. A $w$-assignment $g$ is a function assigning a member of the domain $M_w$ to each variable in the language. 
In an adequate frame, any $w$-assignment can be seen as a $v$-assignment as long as $wRv$, because $M_w \subseteq M_v$ and hence there is a trivial inclusion (or coercion) $\iota_{w, v} : M_w \to M_v$. If $g$ is a such a $w$-assignment, we represent the corresponding $v$-assignment $\iota_{w, v} \circ g$ by $g^\iota$ when $w$ and $v$ are clear from the context.

Two $w$-assignments $g$ and $h$ are \emph{$\Gamma$-alternative}, denoted by $g \xaltern{\Gamma} h$, if they coincide on all variables other than the ones in $\Gamma$. If $\Gamma = \{x\}$, we write $x$-alternative and $g \xaltern{x} h$.

We extend a given $w$-assignment $g$ to terms by defining $g(c) := c^{I_w}$ where $c$ is any constant.

\begin{definition}
  Let $\M = \la W, R, \{M_w\}_{w \in W}, \{I_w\}_{w \in W}, \{J_w\}_{w \in W} \ra$ be a relational model in some signature $\Sigma$, and let $w \in W$ be a world, $g$ be a $w$-assignment, $S$ be an $n$-ary relation symbol, and $\varphi, \psi$ be formulas in the language of $\Sigma$.

  We define $\M, w \Vdash^g \varphi$ ($\varphi$ is true at $w$ under $g$) by induction on $\varphi$ as follows.
  \begin{itemize}

    \item $\M, w \Vdash^g \top$;

    \item $\M, w \Vdash^g S(t_0, \hdots, t_{n-1})$ iff $\la g(t_0), \hdots, g(t_{n-1}) \ra \in S^{J_w}$;

    \item $\M, w \Vdash^g \varphi \land \psi$ iff both $\M, w \Vdash^g \varphi$ and $\M, w \Vdash^g \psi$;

    \item $\M, w \Vdash^g \Diamond \varphi$ iff there is $v \in W$ such that $wRv$ and $\M, v \Vdash^{g^\iota} \varphi$;

    \item $\M, w \Vdash^g \Forall{x} \varphi$ iff for all $w$-assignments $h$ such that $h \xaltern{x} g$, $\M, w \Vdash^{h} \varphi$.

  \end{itemize}
\end{definition}

We now present a number of simple results needed to prove the relational soundness of $\QRC$. These are standard observations about either first-order models or Kripke models that we adapted to our case.

\begin{remark}
\label{rem:assignment:absent_variables}
  Let $\M$ be an adequate model, $w$ be any world, $g, h$ be any $\Gamma$-alternative $w$-assignments, and $\varphi$ be a formula with no free variables in $\Gamma$. Then:
  \begin{equation*}
    \M, w \Vdash^g \varphi \iff \M, w \Vdash^{h} \varphi
    .
  \end{equation*}
\end{remark}

\begin{lemma}[Substitution in formula]
\label{lem:assignment:formula_substitution}
  Let $\M$ be an adequate model, $w$ be a world, and $g, \tilde{g}$ be $x$-alternative $w$-assignments such that $\tilde{g}(x) = g(t)$. Then for every formula $\varphi$ with $t$ free for $x$:
  \begin{equation*}
    \M, w \Vdash^{\tilde{g}} \varphi
    \iff
    \M, w \Vdash^g \varphi\subst{x}{t}
    .
  \end{equation*}
\end{lemma}
\begin{proof}
  By induction on $\varphi$. We only present the cases of the diamond and of the universal quantifier; the remaining cases are straightforward. We assume without loss of generality that $x$ is a free variable of $\varphi$, since otherwise we could use Remark~\ref{rem:assignment:absent_variables}.
  

  Suppose that $\varphi$ is $\Diamond \psi$ and assume that $\M, w \Vdash^{\tilde{g}} \Diamond \psi$. Then there is a world $v$ such that $wRv$ and $\M, v \Vdash^{\tilde{g}^\iota} \psi$. Note that $g^\iota \xaltern{x} \tilde{g}^\iota$ and $\tilde{g}^\iota(x) = g^\iota(t)$ (either $t$ is a variable and this is a consequence of $\tilde{g}(x) = g(t)$, or $t$ is a constant and this follows from $t^{I_w} = t^{I_v}$) and thus by the induction hypothesis $\M, v \Vdash^{g^\iota} \psi\subst{x}{t}$. This gives us $\M, w \Vdash^g \Diamond \psi\subst{x}{t}$, as desired. The other direction is analogous.

  Suppose now that $\varphi = \Forall{z} \psi$ and assume that $\M, w \Vdash^{\tilde{g}} \Forall{z} \psi$. Note that $x$ and $z$ are different variables, for otherwise $x$ would not be free in $\varphi$. Let $h$ be any $w$-assignment such that $h \xaltern{z} g$. We wish to show $\M, w \Vdash^h \psi\subst{x}{t}$. Define $\tilde{h}$ such that $\tilde{h} \xaltern{x} h$ and $\tilde{h}(x) := h(t)$. 
  Then by the induction hypothesis we can reduce our goal to $\M, w \Vdash^{\tilde{h}} \psi$. By our assumption, it is enough to check that $\tilde{h} \xaltern{z} \tilde{g}$.

  In order to see this, note first that $\tilde{h} \xaltern{\{x, z\}} h$ (because $\tilde{h} \xaltern{x} h$). Similarly, $h \xaltern{\{x, z\}} g$ and $g \xaltern{\{x, z\}} \tilde{g}$. Then $\tilde{h} \xaltern{\{x, z\}} \tilde{g}$ by transitivity of $\xaltern{\{x, z\}}$. But $\tilde{g}(x) = g(t)$ by assumption; $g(t) = h(t)$ because $g \xaltern{z} h$ ($z$ and $t$ are not the same variable because otherwise $t$ would not be free for $x$ in $\varphi$); and $h(t) = \tilde{h}(x)$ by construction of $\tilde{h}$. Thus $\tilde{g}(x) = \tilde{h}(x)$, and $\tilde{h} \xaltern{z} \tilde{g}$.

  Towards the other direction, assume that $\M, w \Vdash^g (\Forall{z} \psi)\subst{x}{t}$ and that $x$ and $z$ are not the same variable. Let $\tilde{h} \xaltern{z} \tilde{g}$ be a $w$-assignment. We wish to show $\M, w \Vdash^{\tilde{h}} \psi$. Define $h \xaltern{x} \tilde{h}$ such that $h(x) := g(x)$. Note that $h \xaltern{z} g$ by the transitivity of ${\xaltern{x, z}}$ (using a similar argument to the one above). Thus we know that $\M, w \Vdash^h \psi\subst{x}{t}$ by assumption. It only remains to show that $\tilde{h}(x) = h(t)$, as we can then finally use the induction hypothesis to finish.
  If $t$ is $x$ there is nothing to show, and $t$ cannot be $z$, because $z$ is not free for $x$ in $\Forall{z} \psi$. Thus, $h(t) = g(t) = \tilde{g}(x) = \tilde{h}(x)$.
\end{proof}

We now wish to provide counterparts to Remark~\ref{rem:assignment:absent_variables} and Lemma~\ref{lem:assignment:formula_substitution} for when the change happens in the interpretation of a constant instead of a variable. They are needed to show the soundness of Rule~\ref{rule:constants}. It is straightforward to check that the interpretation of constants not appearing in a formula is not relevant for the truth of that formula:

\begin{remark}
\label{rem:absent_constants}
Let $\M$ and $\M'$ be adequate models differing only in their constant interpretations $\{I_w\}_{w \in W}$ and $\{I'_w\}_{w \in W}$. Let $w$ be any world, $g$ be any $w$-assignment, and $\varphi$ be a formula whose constants are interpreted in the same way by both $\M$ and $\M'$. Then
  \begin{equation*}
    \M, w \Vdash^g \varphi \iff \M', w \Vdash^g \varphi
    .
  \end{equation*}
\end{remark}

However, we need a bit of work to be able to state a counterpart of Lemma~\ref{lem:assignment:formula_substitution} for constants. We want to be able to replace the interpretation of a constant by an element of the domain of some world $w$, but this element may not exist in the domains of the worlds below $w$. Thus we need to first get rid of that part of the model and keep only the sub-graph rooted at $w$.

\begin{definition}
  Given a frame $\F = \la W, R, \{M_w\}_{w \in W} \ra$ and a world $r \in W$, the \emph{frame restricted at $r$}, written $\F\rest{r} = \la W\rest{r}, R\rest{r}, \{M_w\}_{w \in W\rest{r}} \ra$, is defined as the restriction of $\F$ to the world $r$ and all the worlds accessible from $r$ by $R$. Thus, $W\rest{r} := \{r\} \cup \{w \in W \ | \ rRw\}$, and the relation $R\rest{r}$ is $R$ restricted to $W\rest{r}$.

  If $\M = \la \F, \{I_w\}_{w \in W}, \{J_w\}_{w \in W}\ra$ is a model, then $\M\rest{r}$ is defined as $\la \F\rest{r}, \{I_w\}_{w \in W\rest{r}},\allowbreak \{J_w\}_{w \in W\rest{r}}\ra$.
\end{definition}

\begin{remark}
  If $\F$ is an adequate frame, then so is $\F\rest{r}$ for any $r \in W$. Furthermore, if $\M$ is an adequate model, then so is $\M\rest{r}$.
\end{remark}

\begin{remark}
\label{rem:essential_truncated_model}
  Given an adequate model $\M$ and a world $r \in W$, we have that for any formula $\varphi$, any world $w \in W_r$ and any $w$-assignment $g$:
  \begin{equation*}
    \M, w \Vdash^g \varphi \iff \M\rest{r}, w \Vdash^g \varphi.
  \end{equation*}
\end{remark}

\begin{definition}
  Given an adequate model $\M = \la \F, \{I_w\}_{w \in W}, \{J_w\}_{w \in W} \ra$, a world $r \in W$, a constant $c$, and an element of the domain $d \in M_r$, we define $\M\rest{r}\subst{c}{d} := \la \F\rest{r}, \{I'_w\}_{w \in W\rest{r}}, \{J_w\}_{w \in W\rest{r}} \ra$ such that its frame is $\F$ truncated at $r$, the relational symbols interpretation and the interpretation of all constants except for $c$ coincides with that of $\M\rest{r}$, and the interpretation $c^{I'_w}$ of the constant $c$ is $d$ for every $w \in W\rest{r}$.
\end{definition}

\begin{lemma}
\label{lem:model_substitution}
Given a constant $c$, a formula $\varphi$ where $c$ does not appear, an adequate model $\M$, a world $w$, and a $w$-assignment $g$, we have:
  \begin{equation*}
    \M, w \Vdash^{g} \varphi
    \iff
    \M\rest{w}\subst{c}{g(x)}, w \Vdash^{g} \varphi\subst{x}{c}
    .
  \end{equation*}
\end{lemma}
\begin{proof}
  We proceed by induction on the formula $\varphi$. The cases of $\top$, relational symbols, and conjunction are trivial. We assume that $x$ is free in $\varphi$, for otherwise we could use Remarks~\ref{rem:absent_constants} and \ref{rem:essential_truncated_model}.


  Consider the diamond case. If $\M, w \Vdash^{g} \Diamond \psi$, then there is a world $v$ such that $wRv$ and $\M, v \Vdash^{g^\iota} \psi$. By the induction hypothesis we obtain $\M\rest{v}\subst{c}{g^\iota(x)}, v \Vdash^{g^\iota} \psi\subst{x}{c}$. Observe that $\M\rest{v}\subst{c}{g^\iota(x)}$ is the same model as $(\M\rest{w}\subst{c}{g(x)})\rest{v}$, since they share the same frame, the same constant interpretation (because $g(x) = g^\iota(x)$) and the same relational symbol interpretation. Then by Remark~\ref{rem:essential_truncated_model} we get $\M\rest{w}\subst{c}{g(x)}, v \Vdash^{g^\iota} \psi\subst{x}{c}$ and consequently $\M\rest{w}\subst{c}{g(x)} \Vdash^g \Diamond \psi\subst{x}{c}$, as desired. The other implication is analogous.

  Finally, let $\varphi = \Forall{z} \psi$ and assume that $\M, w \Vdash^g \Forall{z} \psi$. Let $h \xaltern{z} g$ be a $w$-assignment, and set out to prove $\M\rest{w}\subst{c}{g(x)}, w \Vdash^h \psi\subst{x}{c}$ (note that $z$ and $x$ are not the same variable for otherwise $x$ would not be free in $\varphi$). Since $h \xaltern{z} g$, we know that $g(x) = h(x)$, so by the induction hypothesis it is enough to show $\M, w \Vdash^h \psi$, which follows from our assumption. The other implication is analogous.
\end{proof}

We are finally ready to prove that $\QRC$ is sound with respect to the relational semantics presented above.

\begin{theorem}[Relational soundness]
\label{thm:modal_soundness}
If $\varphi \vdash \psi$, then for any adequate model $\M$, for any world $w \in W$, and for any $w$-assignment $g$:
  \begin{gather*}
    \M, w \Vdash^g \varphi
    \implies
    \M, w \Vdash^g \psi
    .
  \end{gather*}
\end{theorem}
\begin{proof}
  By induction on the proof of $\varphi \vdash \psi$.

  In the case of the axioms $\varphi \vdash \top$ and $\varphi \vdash \varphi$, the result is clear, as it is for the conjunction elimination axioms.
  The conjunction introduction and cut rules follow easily from the definitions.

  For the necessitation rule assume the result for $\varphi \vdash \psi$ and further assume that $\M, w \Vdash^g \Diamond \varphi$. Then there is a world $v$ such that $wRv$ and $\M, v \Vdash^{g^\iota} \varphi$. We wish to see $\M, w \Vdash^g \Diamond \psi$. Taking $v$ as a suitable witness, our goal changes to $\M, v \Vdash^{g^\iota} \psi$. Thus the induction hypothesis for $v$ and $g^\iota$ finishes the proof.

  For the transitivity axiom, $\Diamond \Diamond \varphi \vdash \Diamond \varphi$, assume that $\M, w \Vdash^g \Diamond \Diamond \varphi$. Then there is a world $v$ such that $wRv$ and $\M, v \Vdash^{\iota_{w, v} \circ g} \Diamond \varphi$, and also a subsequent world $u$ such that $vRu$ and $\M, u \Vdash^{\iota_{v, u} \circ (\iota_{w, v} \circ g)} \varphi$. Observing that $\iota_{v, u} \circ (\iota_{w, v} \circ g)$ is the same as $\iota_{w, u} \circ g$, we get $\M, u \Vdash^{\iota_{w, u} \circ g} \psi$, and the transitivity of $R$ provides $wRu$, which is enough to see $\M, w \Vdash^{g} \Diamond \varphi$, as desired.

  In the case of $\Diamond \Forall{x} \varphi \vdash \Forall{x} \Diamond \varphi$, assume that $\M, w \Vdash^g \Diamond \Forall{x} \varphi$. Then there is $v \in W$ such that $wRv$ and for every $v$-assignment $h$ with $h \xaltern{x} g^\iota$ we have $\M, v \Vdash^{h} \varphi$. Let $f$ be any $w$-assignment such that $f \xaltern{x} g$. Taking $v$ as a suitable world seen by $w$, we wish to check that $\M, v \Vdash^{f^\iota} \varphi$. By assumption, it is enough to see $f^\iota \xaltern{x} g^\iota$, and this follows from $f \xaltern{x} g$.

  For the $\forall$-introduction rule on the right, assume the result for $\varphi \vdash \psi\subst{x}{y}$ with $x \not\in \fv(\varphi)$ and $y$ free for $x$ in $\psi$. Assume further that $\M, w \Vdash^g \varphi$. Let $h$ be a $w$-assignment such that $h \xaltern{x} g$. We wish to see that $\M, w \Vdash^h \psi$. Since $x$ is not a free variable in $\varphi$, we know that $\M, w \Vdash^h \varphi$ by Remark~\ref{rem:assignment:absent_variables}. The result follows from the induction hypothesis with $w$-assignment $h$.

  Consider now the $\forall$-introduction rule on the left. Assume the result for $\varphi\subst{x}{t} \vdash \psi$ with $t$ free for $x$ in $\varphi$ and assume further that $\M, w \Vdash^g \Forall{x} \varphi$. Then for every $w$-assignment $h$ such that $h \xaltern{x} g$ we have $\M, w \Vdash^{h} \varphi$. Define $h \xaltern{x} g$ such that $h(x) = g(t)$. We obtain $\M, w \Vdash^g \psi$ by the induction hypothesis and Lemma~\ref{lem:assignment:formula_substitution}.

  The term instantiation rule, that if $\varphi \vdash \psi$ then $\varphi\subst{x}{t} \vdash \psi\subst{x}{t}$ with $t$ free for $x$ in $\varphi$ and in $\psi$, is sound by Lemma~\ref{lem:assignment:formula_substitution}.

  Finally, consider the generalization on constants rule and assume the result for $\varphi\subst{x}{c} \vdash \psi\subst{x}{c}$, where $c$ does not appear in $\varphi$ nor in $\psi$. Assume further that $\M, w \Vdash^g \varphi$. By Lemma~\ref{lem:model_substitution} we know that $\M\rest{w}\subst{c}{g(x)}, w \Vdash^g \varphi\subst{x}{c}$, and thus by the induction hypothesis that $\M\rest{w}\subst{c}{g(x)}, w \Vdash^g \psi\subst{x}{c}$. This allows us to conclude $\M, w \Vdash^g \psi$ by the same lemma.
\end{proof}

\section{Relational completeness}

We now wish to prove the relational completeness of $\QRC$. For every underivable sequent we provide a model that doesn't satisfy it. These models are term models where the worlds are akin to maximal consistent sets. However, since we have no way to express negative formulas, each world is a pair of sets of formulas instead: the set of positive formulas at that world and the set of negative ones.

We start by defining some notions about pairs of formulas, and we write $p, q, \hdots$ to refer to generic pairs that may not have all the necessary properties to be a world in a term model. Given a pair of sets $p$, the first set is the positive set, or $p^+$, and the second one is the negative set, or $p^-$.

\begin{definition}
  Given a set of formulas $\Gamma$ and a formula $\varphi$, we say that $\varphi$ follows from $\Gamma$, and write $\Gamma \vdash \varphi$, if there are formulas $\gamma_0, \hdots, \gamma_{n} \in \Gamma$ such that $\gamma_0 \land \cdots \land \gamma_{n} \vdash \varphi$.
\end{definition}

\begin{definition}
  Let $\Phi$ be a set of formulas.
  \begin{itemize}
    \item A $\Phi$-extension of a pair $p = \pair{p^+}{p^-}$ is a pair $q = \pair{q^+}{q^-}$ such that $p^+ \subseteq q^+ \subseteq \Phi$ and $p^- \subseteq q^- \subseteq \Phi$. In that case we write $p \subseteq q \subseteq \Phi$.

    \item A pair $p$ is consistent if for every $\delta \in p^-$  we have $p^+ \not\vdash \delta$.

    \item A pair $p \subseteq \Phi$ is $\Phi$-maximal consistent if it is consistent and there is no consistent $\Phi$-extension of $p$.

    \item A pair $p$ is fully witnessed if for every formula $\Forall{x} \varphi \in p^-$ there is a constant $c$ such that $\varphi\subst{x}{c} \in p^-$.

    \item A pair $p$ is $\Phi$-MCW if it is $\Phi$-maximal consistent and fully witnessed.
  \end{itemize}
\end{definition}

\begin{lemma}
\label{lem:maximal_consistentP}
  A pair $p$ is $\Phi$-maximal consistent if and only if it is consistent and for every $\varphi \in \Phi$ either $\varphi \in p^+$ or $\varphi \in p^-$.
\end{lemma}
\begin{proof}
  The right-to-left implication is obvious. To check the other one assume that $p$ is $\Phi$-maximal consistent and let $\varphi \in \Phi$. If $p^+ \vdash \varphi$, then $\pair{p^+ \cup \{\varphi\}}{p^-}$ is still consistent, and thus by maximality it must be that $\varphi \in p^+$. If on the other hand $p^+ \not\vdash \varphi$, then $\pair{p^+}{p^- \cup \{\varphi\}}$ is consistent, and thus we may conclude $\varphi \in p^-$.
\end{proof}


\begin{definition}
  Given a set of constants $C$, the closure of a formula $\varphi$ under $C$, written $\Cl_C(\varphi)$, is defined by induction on the formula as such: $\Cl_C(\top) := \{\top\}$; $\Cl_C(S(t_0, \hdots, t_{n-1}) := \{S(t_0, \hdots, t_{n-1})), \top\}$; $\Cl_C(\varphi \land \psi) := \{\varphi \land \psi\} \cup \Cl_C(\varphi) \cup \Cl_C(\psi)$; $\Cl_C(\Diamond \varphi) := \{\Diamond \varphi\} \cup \Cl_C(\varphi)$; and
  \begin{equation*}
    \Cl_C(\Forall{x} \varphi) :=
      \{\Forall{x} \varphi\}
      \cup
      \bigcup_{c \in C} \Cl_C(\varphi\subst{x}{c})
    .
  \end{equation*}

  The closure under $C$ of a set of formulas $\Gamma$ is the union of the closures under $C$ of each of the formulas in $\Gamma$:
  \begin{equation*}
    \Cl_C(\Gamma) := \bigcup_{\gamma \in \Gamma} \Cl_C(\gamma).
  \end{equation*}
  The closure of a pair $p$ is defined as the closure of $p^+ \cup p^-$.
\end{definition}

  Note that the closure of a set of closed formulas is itself a set of closed formulas.
We often use the concept of closure under a set of constants on an already $\Phi$-maximal pair when we wish to extend the signature of the formulas in $\Phi$ with a new set of constants.

Given a consistent pair $p$, we wish to generate a $\Phi$-maximal consistent and fully witnessed extension of $p$, for some set of formulas $\Phi$. In the usual Henkin construction this is traditionally accomplished in two steps: first extend the signature to include a constant for each existential statement and add every closed formula of the form $\exists{x} \varphi \to \varphi\subst{x}{c_\varphi}$ to your set, proving that this didn't break consistency. Then prove a Lindenbaum lemma to the effect that consistent sets can be extended to maximal consistent sets. The resulting sets will be maximal, consistent, and fully witnessed. However we can not do this because we cannot express implications. Thus if we were to add a witness for every existential formula in our original pair $p$ (read: universal formula in $p^-$) and then use a Lindenbaum lemma to make it maximal, there could be new existential formulas without witnesses. We might have to iterate the process over and over again, or at least a proof of termination would be non-trivial.
Fortunately, this isn't needed. We can manage with a finite set of witnesses, as is shown by the following lemma.

\begin{lemma}
\label{lem:lindenbaum}
Given a finite signature $\Sigma$ with constants $C$, a finite set of closed formulas $\Phi$ in the language of $\Sigma$ and a consistent pair $p \subseteq \Cl_{C}(\Phi)$, there is a finite set of constants $D \supseteq C$ and a pair $q \supseteq p$ in the language of $\Sigma$ extended by $D$ such that $q$ is $\Cl_{D}(\Phi)$-MCW and $\mdepth(q^+) = \mdepth(p^+)$.
\end{lemma}
\begin{proof}
  Let $N := \{c_0, \hdots, c_{\udepth(\Phi)-1}\}$ and $D := C \cup N$.

  Let $q_0 := p$. For every formula $\varphi_i$ in $\Cl_D(\Phi)$, if $p^+ \vdash \varphi_i$, define $q_{i+1} = \pair{q^+_i \cup \{\varphi_i\}}{q^-_i}$; otherwise define $q_{i+1} = \pair{q^+_i}{q^-_i \cup \{\varphi_i\}}$. Let $q := q_n$, where $n$ is the size of $\Cl_D(\Phi)$, i.e., $q$ is what we have at the final iteration of this process.

  Now assume by way of contradiction that $q$ is not consistent, and let $\psi \in q^-$ be such that $q^+ \vdash \psi$. Note that for every $\chi \in q^+$ we know that $p^+ \vdash \chi$, because this was the required condition to add $\chi$ to $q^+$ in the first place. Thus, it must be that $p^+ \vdash \psi$. But then the algorithm would have placed $\psi$ in $q^+$ instead of $q^-$ and we reach a contradiction. We conclude that $q$ is consistent.

  Lemma~\ref{lem:maximal_consistentP} tells us that $q$ is $\Cl_D(\Phi)$-maximal consistent, because every formula of $\Cl_D(\Phi)$ is either in $q^+$ or $q^-$.

  On the other hand, we know by Lemma~\ref{lem:mdepth} that $\mdepth(q^+) \leq \mdepth(p^+)$ because every formula in $q^+$ is a consequence of $p^+$. We obtain the equality by observing that $p^+ \subseteq q^+$.

  It remains to show that $q$ is fully witnessed. Let $\Forall{x} \psi$ be a formula in $q^-$. We claim that there is $c_i \in N$ such that $c_i$ does not appear in $\Forall{x} \psi$. The constants in $N$ are new, so the only way to have a formula $\chi \in \Cl_D(\Phi)$ with a constant $c_j \in N$ is if the formula $\Forall{y} \chi\subst{c_j}{y}$ 
  is also in $\Cl_D(\Phi)$, for some variable $y$ that does not appear (free) in $\chi$. Assume then that all the constants in $N$ appear in $\psi$. Then the formula $\Forall{y_0} \cdots \Forall{y_{m-1}} \Forall{x} \psi\subst{c_{s_{m-1}}}{y_{m-1}} \cdots \subst{c_{s_0}}{y_0}$ must be in $\Cl_D(\Phi)$ for some variables $y_i$ and permutation $s$ of the numbers between $0$ and $m-1$. But this formula has quantifier depth $m+1$, which is a contradiction because the closure under any set of constants doesn't change the depth of a set of formulas.

  Let then $\Forall{x} \psi \in q^-$ and $c_i \in N$ be a constant that does not appear in $\Forall{x} \psi$. Then we claim that $\psi\subst{x}{c_i} \in q^-$. Assume it is not the case. Then it must be that $p^+ \vdash \psi\subst{x}{c_i}$. Note that $c_i$ does not appear in $p^+$ and that $x$ is not a free variable of $p^+$ due to it being a set of closed formulas. Then by Lemma~\ref{item:constants_forall} we obtain that $p^+ \vdash \Forall{x} \psi$, which is a contradiction.
\end{proof}

The next step is to link maximal consistent and fully witnessed pairs together through a relation that respects the diamond formulas in the pair. To that end we define $\hat{R}$ and prove some properties about it.

\begin{definition}
  The relation $\hat{R}$ between pairs is such that $p\hat{R}q$ if and only if both of following hold:
  \begin{enumerate}[label=\upshape(\roman*)]
    \item for any formula $\Diamond \varphi \in p^-$ we have $\varphi, \Diamond \varphi \in q^-$; and
    \item there is some formula $\Diamond \psi \in p^+ \cap q^-$.
  \end{enumerate}
\end{definition}

\begin{lemma}
\label{lem:hatR-transitive-irreflexive}
  The relation $\hat{R}$ restricted to consistent pairs is transitive and irreflexive.
\end{lemma}
\begin{proof}
  In order to see that $\hat{R}$ is transitive, assume that $p\hat{R}q\hat{R}r$. We wish to see that $p\hat{R}r$. Let $\Diamond \varphi \in p^-$ be arbitrary. Then $\Diamond \varphi \in q^-$ because $p\hat{R}q$, and then $\varphi, \Diamond \varphi \in r^-$ because $q\hat{R}r$.
  Let now $\Diamond \psi \in {p^+ \cap q^-}$. Since $q\hat{R}r$ we know that $\Diamond \psi \in r^-$. Then $\Diamond \psi \in p^+ \cap r^-$.

  Regarding irreflexivity, suppose that there is a pair $p$ such that $p\hat{R}p$. Then there must be $\Diamond \psi \in p^+ \cap p^-$, which contradicts the consistency of $p$.
\end{proof}

There is an equivalent formulation of $\hat{R}$ by looking at the positive sets.

\begin{lemma}
\label{lem:R+}
Given a set of formulas $\Phi$, two sets of constants $C \subseteq D$, and pairs $p, q$ such that $p$ is $\Cl_C(\Phi)$-maximal consistent and $q$ is $\Cl_D(\Phi)$-maximal consistent, we have that $p\hat{R}q$ if and only if both of the following hold:
    \begin{enumerate}[label=\upshape(\roman*)]
      \item for every formula $\Diamond \varphi \in \Cl_C(\Phi)$, if either $\varphi \in q^+$ or $\Diamond \varphi \in q^+$, then $\Diamond \varphi \in p^+$; and
      \item there is some formula $\Diamond \psi \in p^+ \cap q^-$.
    \end{enumerate}
\end{lemma}
\begin{proof}
  Assume that $p\hat{R}q$ and let $\Diamond \varphi \in \Cl_C(\Phi)$ be such that either $\varphi \in q^+$ or $\Diamond \varphi \in q^+$. Assume by contradiction that $\Diamond \varphi \notin p^+$. Then by Lemma~\ref{lem:maximal_consistentP} we know that $\Diamond \varphi \in p^-$. Thus since $p\hat{R}q$, we obtain both $\varphi \in q^-$ and $\Diamond \varphi \in q^-$. But this contradicts the consistency of $q$. 
  The last condition holds by the definition of $\hat{R}$.

  Assume now that these conditions hold, towards checking that $p\hat{R}q$. Only the first condition is in question. Let $\Diamond \varphi \in p^-$ and assume that $\varphi \notin q^-$. By Lemma~\ref{lem:maximal_consistentP}, it must be that $\varphi \in q^+$. Then $\Diamond \varphi \in p^+$, which contradicts the consistency of $p$. Assume now that $\Diamond \varphi \notin q^-$. By the same token, $\Diamond \varphi$ must be in $q^+$. Then $\Diamond \varphi \in p^+$, reaching a contradiction again.
\end{proof}

The following lemma states that, given a suitable pair $p$ where $\Diamond \varphi$ holds, we can find a second suitable pair $q$ where $\varphi$ holds such that $p\hat{R}q$.

\begin{lemma}[Pair existence]
\label{lem:pair_existence}
  Let $\Sigma$ be a signature with a finite set of constants $C$, and $\Phi$ be a finite set of closed formulas in the language of $\Sigma$. If $p$ is a $\Cl_C(\Phi)$-MCW pair and $\Diamond \varphi \in p^+$, then there is a finite set of constants $D \supseteq C$ and a $\Cl_{D}(\Phi)$-MCW pair $q$ such that $p\hat{R}q$, $\varphi \in q^+$, and $\mdepth(q^+) < \mdepth(p^+)$.
\end{lemma}
\begin{proof}
  Consider the pair $r = \pair{\{\varphi\}}{\{\delta, \Diamond \delta \ | \ \Diamond \delta \in p^-\} \cup \{\Diamond \varphi\}}$. Assume that $r$ is not consistent, and thus that there is a formula $\psi \in r^-$ such that $\varphi \vdash \psi$. It cannot be that $\psi$ is $\Diamond \varphi$ due to Lemma~\ref{cor:A|/-<>A}. Thus there is $\Diamond \delta \in p^-$ such that either $\varphi \vdash \delta$ or $\varphi \vdash \Diamond \delta$. By Rule~\ref{rule:necessitation} we get either $\Diamond \varphi \vdash \Diamond \delta$ or $\Diamond \varphi \vdash \Diamond \Diamond \delta$, which also implies $\Diamond \varphi \vdash \Diamond \delta$ by Axiom~\ref{ax:transitivity}. This contradicts the consistency of $p$, which leads us to conclude that $r$ is consistent.
  
  We can now use Lemma~\ref{lem:lindenbaum} to obtain a finite set of constants $D \supseteq C$ and a $\Cl_{D}(\Phi)$-MCW pair $q \supseteq r$ such that $\mdepth(q^+) = \mdepth(r^+) = \mdepth(\varphi) < \mdepth(p^+)$.

  It remains to show that $p\hat{R}q$, but this is clear by the definition of $r$: for every $\Diamond \delta \in p^-$, the formulas $\delta$ and $\Diamond \delta$ are in $r^-$ (and hence in $q^-$), and the formula $\Diamond \varphi$ is both in $p^+$ and in $q^-$.
\end{proof}

We are now ready to define an adequate model $\M[p]$ from any given finite and consistent pair $p$ such that $\M[p]$ satisfies the formulas in $p^+$ and doesn't satisfy the formulas in $p^-$.
The idea is to build a term model where each world $w$ is a $\Cl_{M_w}(p)$-MCW pair, and the worlds are related by (a sub-relation of) $\hat{R}$. The worlds in this model will be pairs of formulas in different signatures, as we will add new constants every time we create a new world. However, the model is intended to satisfy only formulas in the original signature of $p$.

\begin{definition}
\label{def:Mp}
Given a finite consistent pair $p$ of closed formulas with constants in a finite set $C$, we define an adequate model $\M[p]$.
  
  We start by defining the underlying frame in an iterative manner. The root is given by Lemma~\ref{lem:lindenbaum} applied to $C$ and $p$, obtaining $D$ and $q$. Frame $\F^0$ is then defined such that its set of worlds is $W^0 := \{q\}$, its relation $R^0$ is empty, and the domain of $q$ is $M^0_{q} := D$.

  Assume now that we already have a frame $\F^i$, and we set out to define $\F^{i+1}$ as an extension of $\F^i$. For each leaf $w$ of $\F^i$, i.e., each world such that there is no world $v \in \F^i$ with $w R^i v$, and for each formula $\Diamond \varphi \in w^+$, use Lemma~\ref{lem:pair_existence} to obtain a finite set $E \supseteq M^i_w$ and a $\Cl_E(w)$-MCW pair $v$ such that $w\hat{R}v$, $\varphi \in v^+$, and $\mdepth(v^+) < \mdepth(w^+)$. Now add $v$ to $W^{i+1}$, add $\la w, v \ra$ to $R^{i+1}$, and define $M^{i+1}_v$ as $E$.

  The process described above terminates because each pair is finite and the modal depth of $p^+$ (and consequently of $\Cl_X(p)$ for any set $X$) is also finite. Thus there is a final frame $\F^{\mdepth(p^+)}$. This frame is inclusive by construction, but not transitive. We obtain $\F[p]$ as the transitive closure of $\F^{\mdepth(p^+)}$, which can be easily seen to still be inclusive. Thus the frame $\F[p]$ is adequate.

  In order to obtain the model $\M[p]$ based on the frame $\F[p]$, let $I_q$ take constants in $C$ to their corresponding version as domain elements and if $w$ is any other world, let $I_w$ coincide with $I_q$. This is necessary to make sure that the model is concordant, because $q$ sees every other world, and is sufficient to see that $\M[p]$ is adequate.
  Finally, given an $n$-ary predicate letter $S$ and a world $w$, define $S^{J_w}$ as the set of $n$-tuples $\la d_0, \hdots, d_{n-1}\ra \subseteq (M_w)^n$ such that $S(d_0, \hdots, d_{n-1}) \in w^+$.
\end{definition}

\begin{lemma}
  Let $p$ be as above. The following are properties of $\F[p] = \la W, R,\allowbreak \{M_w\}_{w \in W}\ra$ and $\M[p] = \la \F[p], \{I_w\}_{w \in W}, \{J_w\}_{w \in W} \ra$:
  \begin{enumerate}[label=\upshape(\roman*),ref=\thetheorem.(\roman*)]
    \item Every world $w \in W$ is $\Cl_{M_w}(p)$-maximal consistent and fully witnessed.
    \item For every world $w \in W$, we have $\top \in w^+$.
    \item For any two worlds $w, u \in W$, if $wRu$, then $w\hat{R}u$. \label{item:RimplieshatR}

  \end{enumerate}
\end{lemma}
\begin{proof}
  These are simple consequences of the definition of $\M[p]$. For the last one, note that $R$ is the transitive closure of $R^{\mdepth(p^+)}$. If $wR^{\mdepth(p^+)}u$, then $w\hat{R}u$ by construction. The result then follows by the transitivity of $\hat{R}$ (Lemma~\ref{lem:hatR-transitive-irreflexive}).
\end{proof}

We are almost ready to state the truth lemma, which roughly states that provability at a world $w$ of $\M[p]$ is the same as membership in $w^+$. However, the signatures of the worlds of $\M[p]$ are more expressive than the signature of the formulas we care about. Furthermore, all the formulas in the worlds of $\M[p]$ are closed, while formulas in general may have free variables. In order to deal with this, we replace the free variables of a formula with constants in the appropriate signature first.

\begin{definition}
  Given a formula $\varphi$ in a signature $\Sigma$ and a function $g$ from the set of variables to a set of constants in some signature $\Sigma' \supseteq \Sigma$, we define the formula $\varphi^g$ in the signature $\Sigma'$ as $\varphi$ with each free variable $x$ simultaneously replaced by $g(x)$.
\end{definition}

\begin{lemma}[Truth lemma]
\label{lem:modal_truth}
Let $\Sigma$ be a signature with a finite set of constants $C$. For any finite non-empty consistent pair $p$ of closed formulas in the language of $\Sigma$, world $w \in \M[p]$, $w$-assignment $g$, and formula $\varphi$ in the language of $\Sigma$ such that $\varphi^g \in \Cl_{M_w}(p)$, we have that
  \begin{equation*}
    \M[p], w \Vdash^g \varphi
      \iff
      \varphi^g \in w^+
    .
  \end{equation*}
\end{lemma}
\begin{proof}
  By induction on $\varphi$. The cases of $\top$ and conjunction are straightforward, so we focus on the other ones.


      In the case of the relational symbols, we can take $\varphi = S(x, c)$ without loss of generality, where $c \in C$. Note that $\M[p], w \Vdash^g S(x, c)$ if and only if $\la g(x), c^{I_w} \ra \in S^{J_w}$, if and only if $S(g(x), c^{I_w}) \in w^+$. Since $c \in C$, we know by the definition of $\M[p]$ that $c^{I_w} = c$. Thus, we conclude that $\M[p], w \Vdash^g S(x, c)$ if an only if $S(g(x), c) \in w^+$, as desired.



      Consider now the case of the universal quantifier. For the left to right implication, suppose that $\M[p], w \Vdash^g \Forall{x} \varphi$. Then for every $w$-assignment $h \xaltern{x} g$ we have $\M[p], w \Vdash^h \varphi$. Thus for each such $h$ we know that $\varphi^h \in w^+$ by the induction hypothesis ($\varphi^h \in \Cl_{M_w}(p)$ because $(\Forall{x} \varphi)^g \in \Cl_{M_w}(p)$). We want to show that $(\Forall{x} \varphi)^g \in w^+$, i.e., that $\Forall{x} \varphi^{g \backslash x} \in w^+$. Assume by contradiction that this is not the case. Then, since $w$ is $\Cl_{M_w}(p)$-maximal consistent, it must be that $\Forall{x} \varphi^{g \backslash x} \in w^-$. Let $c \in M_w$ be a witness such that $\varphi^{g \backslash x}\subst{x}{c} \in w^-$, which exists because $w$ is fully witnessed. Let $h$ be the $w$-assignment that coincides with $g$ everywhere except at $x$, where $h(x) = c$. Then $g \xaltern{x} h$ and $\varphi^{g \backslash x}\subst{x}{c} = \varphi^h$. But this contradicts our earlier observation that for every such $h$ the formula $\varphi^h$ is in $w^+$.

      For the right to left implication, let $\Forall{x} \varphi^{g \backslash x} \in w^+$, and let $h \xaltern{x} g$ be any $w$-assignment. We want to show that $\M[p], w \Vdash^h \varphi$. By the induction hypothesis this is the same as showing that $\varphi^h \in w^+$. But $\varphi^h = \varphi^{g \backslash x}\subst{x}{h(x)}$, and this is in $w^+$ by the completeness and consistency of $w$.

      Finally, consider the case of the diamond. For the left to right implication, assume that $\M[p], w \Vdash^g \Diamond \varphi$. Then there is some world $u$ such that $wRu$ and $\M[p], u \Vdash^{g^\iota} \varphi$. By the induction hypothesis we obtain $\varphi^{g^\iota} \in u^+$, and consequently $\varphi^g \in u^+$. Now, since $wRu$, we also know that $w\hat{R}u$ by Lemma~\ref{item:RimplieshatR}, and thus by Lemma~\ref{lem:R+} we obtain $\Diamond \varphi^g \in w^+$ as desired.

      For the right to left implication, assume that $(\Diamond \varphi)^g \in w^+$. By the construction of $\M[p]$, there is a world $u$ such that $\varphi^g \in u^+$ (and hence $\varphi^{g^\iota} \in u^+$) and $wRu$, and then $\M[p], u \Vdash^{g^\iota} \varphi$ by the induction hypothesis, from which we finally conclude $\M[p], w \Vdash^g \Diamond \varphi$.
\end{proof}

\begin{theorem}[Completeness]
  If $\varphi \not\vdash \psi$, then there are an adequate model $\M$, a world $w \in W$, and a $w$-assignment $g$ such that
  \begin{equation*}
    \M, w \Vdash^g \varphi \quad\text{and}\quad \M, w \not\Vdash^g \psi
    .
  \end{equation*}
\end{theorem}
\begin{proof}
  Define a set of new constants $C := \{c_{x_i} \ | \ x_i \in \fv(p)\}$ and let $g$ be a map from the set of variables to $C$ that assigns $c_{x_i}$ to $x_i$ for each $i$. Define $p$ as $\pair{\varphi^g}{\psi^g}$, and assume it is not consistent, i.e., that $\varphi^g \vdash \psi^g$. Then by (a generalization of) Rule~\ref{rule:constants} and Lemma~\ref{lem:signatures} we would get that $\varphi \vdash \psi$. Thus $p$ is consistent.
  Let $\M[p]$ be the model generated from $p$ as in Definition~\ref{def:Mp} and let $w$ be the root of this model, which is an extension of $p$. Lemma~\ref{lem:modal_truth} tells us that $\M[p], w \Vdash^g \varphi$ and $\M[p], w \not\Vdash^g \psi$ because $\varphi^g \in w^+$ and $\psi^g \notin w^+$.
\end{proof}

We conclude by noting that $\QRC$ has the finite model property, and is thus decidable.


\begin{thebibliography}{21}
\expandafter\ifx\csname natexlab\endcsname\relax\def\natexlab#1{#1}\fi
\expandafter\ifx\csname url\endcsname\relax
  \def\url#1{{\tt #1}}\fi
\expandafter\ifx\csname urlprefix\endcsname\relax\def\urlprefix{URL }\fi

\bibitem[{Beklemishev(2004)}]{Beklemishev:2004:ProvabilityAlgebrasAndOrdinals}
Beklemishev, L.~D. (2004).
\newblock Provability algebras and proof-theoretic ordinals, {I}.
\newblock {\em Annals of Pure and Applied Logic\/}, {\em 128\/}, 103--124.

\bibitem[{Beklemishev(2005)}]{Beklemishev:2005:VeblenInGLP}
Beklemishev, L.~D. (2005).
\newblock Veblen hierarchy in the context of provability algebras.
\newblock In P.~H\'ajek, L.~Vald\'es-Villanueva, \& D.~{Westerst\r ahl} (Eds.)
  {\em Logic, Methodology and Philosophy of Science, Proceedings of the Twelfth
  International Congress\/}, (pp. 65--78). Kings College Publications.

\bibitem[{Beklemishev(2012)}]{Beklemishev2012}
Beklemishev, L.~D. (2012).
\newblock Calibrating provability logic: From modal logic to {R}eflection
  {C}alculus.
\newblock In T.~Bolander, T.~Braüner, T.~S. Ghilardi, \& L.~Moss (Eds.) {\em
  Advances in Modal Logic 9\/}, (pp. 89--94). London: College Publications.

\bibitem[{Beklemishev(2014)}]{Beklemishev:2014:PositiveProvabilityLogic}
Beklemishev, L.~D. (2014).
\newblock Positive provability logic for uniform reflection principles.
\newblock {\em Annals of Pure and Applied Logic\/}, {\em 165\/}(1), 82--105.

\bibitem[{Beklemishev et~al.(2014)Beklemishev, Fern\'andez-Duque, \&
  Joosten}]{BeklemishevFernandezJoosten:2014:LinearlyOrderedGLP}
Beklemishev, L.~D., Fern\'andez-Duque, D., \& Joosten, J.~J. (2014).
\newblock {On provability logics with linearly ordered modalities}.
\newblock {\em Studia Logica\/}, {\em 102\/}, 541--566.

\bibitem[{Beklemishev \&
  Pakhomov(2019)}]{BeklemishevPakhomov:2019:GLPforTheoriesOfTruth}
Beklemishev, L.~D., \& Pakhomov, F.~N. (2019).
\newblock Reflection algebras and conservation results for theories of iterated
  truth.
\newblock {\em arXiv:1908.10302 [math.LO]\/}.

\bibitem[{Boolos(1993)}]{Boolos:1993:LogicOfProvability}
Boolos, G.~S. (1993).
\newblock {\em The {L}ogic of {P}rovability\/}.
\newblock Cambridge: Cambridge University Press.

\bibitem[{Dashkov(2012)}]{Dashkov:2012:PositiveFragment}
Dashkov, E.~V. (2012).
\newblock {On the {P}ositive {F}ragment of the {P}olymodal {P}rovability
  {L}ogic GLP}.
\newblock {\em Mathematical Notes\/}, {\em 91\/}(3-4), 318--333.

\bibitem[{Fern\'andez-Duque \&
  Hermo~Reyes(2019)}]{HermoFernandez:2019:BracketCalculus}
Fern\'andez-Duque, D., \& Hermo~Reyes, E. (2019).
\newblock A self-contained provability calculus for {$\Gamma_0$}.
\newblock In R.~Iemhoff, M.~Moortgat, \& R.~J. G.~B. de~Queiroz (Eds.) {\em
  Logic, Language, Information, and Computation - 26th International Workshop,
  WoLLIC 2019, Utrecht, The Netherlands, July 2-5, 2019, Proceedings\/}, vol.
  11541 of {\em Lecture Notes in Computer Science\/}, (pp. 195--207). Springer.

\bibitem[{Fern\'andez-Duque \&
  Joosten(2013)}]{FernandezJoosten:2013:ModelsOfGLP}
Fern\'andez-Duque, D., \& Joosten, J.~J. (2013).
\newblock Models of transfinite provability logics.
\newblock {\em Journal of Symbolic Logic\/}, {\em 78\/}(2), 543--561.

\bibitem[{Fern\'andez-Duque \&
  Joosten(2014)}]{FernandezJoosten:2014:WellOrders}
Fern\'andez-Duque, D., \& Joosten, J.~J. (2014).
\newblock Well-orders in the transfinite {J}aparidze algebra.
\newblock {\em Logic Journal of the {IGPL}\/}, {\em 22\/}(6), 933--963.

\bibitem[{Fern\'andez-Duque \&
  Joosten(2018)}]{FernandezJoosten:2018:OmegaRuleInterpretationGLP}
Fern\'andez-Duque, D., \& Joosten, J.~J. (2018).
\newblock The omega-rule interpretation of transfinite provability logic.
\newblock {\em Annals of Pure and Applied Logic\/}, {\em 169\/}(4), 333--371.

\bibitem[{Goldblatt(2011)}]{Goldblatt2011}
Goldblatt, R. (2011).
\newblock {\em Quantifiers, propositions and identity, admissible semantics for
  quantified modal and substructural logics\/}.
\newblock Cambridge University Press.

\bibitem[{H\'ajek \& Pudl\'ak(1993)}]{HajekPudlak:1993:Metamathematics}
H\'ajek, P., \& Pudl\'ak, P. (1993).
\newblock {\em Metamathematics of {F}irst {O}rder {A}rithmetic\/}.
\newblock Berlin, Heidelberg, New York: Springer-{V}erlag.

\bibitem[{Hughes \& Cresswell(1996)}]{HughesCresswell1996}
Hughes, G.~E., \& Cresswell, M.~J. (1996).
\newblock {\em A New Introduction to Modal Logic\/}.
\newblock Routledge.

\bibitem[{Japaridze(1988)}]{Japaridze:1988}
Japaridze, G. (1988).
\newblock The polymodal provability logic.
\newblock In {\em Intensional logics and logical structure of theories:
  material from the {F}ourth {S}oviet-{F}innish Symposium on Logic\/}, (pp.
  16--48). Tiblisi: Metsniereba.
\newblock In {R}ussian.

\bibitem[{Kikot et~al.(2019)Kikot, Kurucz, Tanaka, Wolter, \&
  Zakharyaschev}]{Kikot2019}
Kikot, S., Kurucz, A., Tanaka, Y., Wolter, F., \& Zakharyaschev, M. (2019).
\newblock Kripke completeness of strictly positive modal logics over
  meet-semilattices with operators.
\newblock {\em Journal of Symbolic Logic\/}, {\em 84\/}(2), 533--588.

\bibitem[{Kripke(1963)}]{Kripke1963}
Kripke, S.~A. (1963).
\newblock Semantical considerations on modal logic.
\newblock {\em Acta Philosophica Fennica\/}, {\em 16\/}, 83--94.

\bibitem[{Solovay(1976)}]{Solovay:1976}
Solovay, R.~M. (1976).
\newblock Provability interpretations of modal logic.
\newblock {\em Israel Journal of Mathematics\/}, {\em 28\/}, 33--71.

\bibitem[{Vardanyan(1986)}]{Vardanyan1986}
Vardanyan, V.~A. (1986).
\newblock Arithmetic complexity of predicate logics of provability and their
  fragments.
\newblock {\em Doklady Akad. Nauk SSSR\/}, {\em 288\/}(1), 11--14.
\newblock In Russian. English translation in Soviet Mathematics Doklady 33,
  569--572 (1986).

\bibitem[{Visser \& de~Jonge(2006)}]{VisserAndDeJonge:2006:NoEscape}
Visser, A., \& de~Jonge, M. (2006).
\newblock No escape from {V}ardanyan's theorem.
\newblock {\em Archive for Mathematical Logic\/}, {\em 45\/}(5), 539--554.

\end{thebibliography}

\end{document}